\documentclass{article}

\usepackage[english]{babel}
\usepackage{amsmath, amssymb, amsthm, amsfonts, bm}
\usepackage{enumerate}
\usepackage{graphicx, color}
\usepackage[colorlinks=true]{hyperref}
\usepackage[shortlabels]{enumitem}
\usepackage{pgf,tikz}
\usepackage{pgfplots}
\usepackage{tkz-tab}
\usepackage{float}
\usetikzlibrary{shapes,calc,arrows,arrows.meta,backgrounds}
\usetikzlibrary{scopes,svg.path,shapes.geometric,shadows}

\numberwithin{equation}{section}
\newtheorem{theorem}{Theorem}[section]

\newtheorem{corollary}[theorem]{Corollary}
\newtheorem{lemma}[theorem]{Lemma}
\theoremstyle{definition}

\theoremstyle{remark}

\newcommand{\Rnum}[1]{\uppercase\expandafter{\romannumeral #1\relax}}

\newcommand{\mr}[1]{\mathrm{#1}}
\newcommand{\mb}[1]{\mathbb{#1}}
\newcommand{\mc}[1]{\mathcal{#1}}
\newcommand\numberthis{\addtocounter{equation}{1}\tag{\theequation}}

\DeclareMathOperator{\Av}{Av} 

\def\clap#1{\hbox to 0pt{\hss#1\hss}}

\title{On a sharp upper bound related to the Bellman function of three integral variables of the dyadic maximal operator}
\author{Eleftherios N. Nikolidakis}
\date{\today}

\begin{document}
\maketitle

\begin{abstract}
We study properties for the sharp upper bound for integral quantities related to the Bellman function of three integral variables of the dyadic maximal operator, that is determined in \cite{Nikol}.
\end{abstract}

\section{Introduction} \label{sec:0}
The dyadic maximal operator on $\mb R^n$ is a useful tool in analysis and is defined by
\begin{equation} \label{eq:0p1}
	\mc M_d\varphi(x) = \sup\left\{ \frac{1}{|S|} \int_S |\varphi(u)|\,\mr du: x\in S,\ S\subseteq \mb R^n\ \text{is a dyadic cube} \right\},
\end{equation}
for every $\varphi\in L^1_\text{loc}(\mb R^n)$, where $|\cdot|$ denotes the Lebesgue measure on $\mb R^n$, and the dyadic cubes are those formed by the grids $2^{-N}\mb Z^n$, for $N=0, 1, 2, \ldots$.\\
It is well known that it satisfies the following weak type (1,1) inequality
\begin{equation} \label{eq:0p2}
	\left|\left\{ x\in\mb R^n: \mc M_d\varphi(x) > \lambda \right\}\right| \leq \frac{1}{\lambda} \int_{\left\{\mc M_d\varphi > \lambda\right\}} |\varphi(u)|\,\mr du,
\end{equation}
for every $\varphi\in L^1(\mb R^n)$, and every $\lambda>0$,
from which it is easy to get the following  $L^p$-inequality
\begin{equation} \label{eq:0p3}
	\|\mc M_d\varphi\|_p \leq \frac{p}{p-1} \|\varphi\|_p,
\end{equation}
for every $p>1$, and every $\varphi\in L^p(\mb R^n)$.
It is easy to see that the weak type inequality \eqref{eq:0p2} is the best possible. For refinements of this inequality one can consult \cite{6}.

It has also been proved that \eqref{eq:0p3} is best possible (see \cite{1} and \cite{2} for general martingales and \cite{18} for dyadic ones).
An approach for studying the behavior of this maximal operator in more depth is the introduction of the so-called Bellman functions which play the role of generalized norms of $\mc M_d$. Such functions related to the $L^p$-inequality \eqref{eq:0p3} have been precisely identified in \cite{4}, \cite{5} and \cite{11}. For the study of the Bellman functions of $\mc M_d$, we use the notation $\Av_E(\psi)=\frac{1}{|E|} \int_E \psi$, whenever $E$ is a Lebesgue measurable subset of $\mb R^n$ of positive measure and $\psi$ is a real valued measurable function defined on $E$. We fix a dyadic cube  $Q$ and define the localized maximal operator $\mc M'_d\varphi$ as in \eqref{eq:0p1} but with the dyadic cubes $S$ being assumed to be contained in $Q$. Then for every $p>1$ we let
\begin{equation} \label{eq:0p4}
	B_p(f,F)=\sup\left\{ \frac{1}{|Q|} \int_Q (\mc M'_d\varphi)^p: \Av_Q(\varphi)=f,\ \Av_Q(\varphi^p)=F \right\},
\end{equation}
where $\varphi$ is nonnegative in $L^p(Q)$ and the variables $f, F$ satisfy $0<f^p\leq F$. By a scaling argument it is easy to see that \eqref{eq:0p4} is independent of the choice of $Q$ (so we may choose
$Q$ to be the unit cube $[0,1]^n$).
In \cite{5}, the function \eqref{eq:0p4} has been precisely identified for the first time. The proof has been given in a much more general setting of tree-like structures on probability spaces.

More precisely we consider a non-atomic probability space $(X,\mu)$ and let $\mc T$ be a family of measurable subsets of $X$, that has a tree-like structure similar to the one in the dyadic case (the exact definition can be found in  []).
Then we define the dyadic maximal operator associated to $\mc T$, by
\begin{equation} \label{eq:0p5}
	\mc M_{\mc T}\varphi(x) = \sup \left\{ \frac{1}{\mu(I)} \int_I |\varphi|\,\mr \; d\mu: x\in I\in \mc T \right\},
\end{equation}
for every $\varphi\in L^1(X,\mu)$, $x\in X$.

This operator is related to the theory of martingales and satisfies essentially the same inequalities as $\mc M_d$ does. Now we define the corresponding Bellman function of four variables of $\mc M_{\mc T}$, by
\begin{multline} \label{eq:0p6}
	B_p^{\mc T}(f,F,L,k) = \sup \left\{ \int_K \left[ \max(\mc M_{\mc T}\varphi, L)\right]^p\mr \; d\mu: \varphi\geq 0, \int_X\varphi\,\mr \; d\mu=f, \right. \\  \left. \int_X\varphi^p\,\mr \; d\mu = F,\ K\subseteq X\ \text{measurable with}\ \mu(K)=k\right\},
\end{multline}
the variables $f, F, L, k$ satisfying $0<f^p\leq F $, $L\geq f$, $k\in (0,1]$.
The exact evaluation of \eqref{eq:0p6} is given in \cite{5}, for the cases where $k=1$ or $L=f$. In the first case the author (in \cite{5}) precisely identifies the function $B_p^{\mc T}(f,F,L,1)$ by evaluating it in a first stage for the case where $L=f$. That is he precisely identifies $B_p^{\mc T}(f,F,f,1)$ (in fact $B_p^{\mc T}(f,F,f,1)=F \omega_p (\frac{f^p}{F})^p$, where                         $\omega_p: [0,1] \to [1,\frac{p}{p-1}]$ is the inverse function $H^{-1}_p$, of $H_p(z) = -(p-1)z^p + pz^{p-1}$). 

The proof of the above mentioned evaluation relies on a one-parameter integral inequality which is proved by arguments based on a linearization of the dyadic maximal operator. More precisely the author in \cite{5} proves that the inequality

\begin{equation}\label{eq:0p7}
	F\geq \frac{1}{(\beta+1)^{p-1}} f^p + \frac{(p-1)\beta}{(\beta+1)^p} \int_X (M_{\mathcal{T}}\varphi)^p \; d\mu,
\end{equation}
is true for every non-negative value of the parameter $\beta$ and sharp for one that depends on $f$, $F$ and $p$, namely for $\beta=\omega_p (\frac{f^p}{F})-1$. This gives as a consequence an upper bound for $B_p^{\mc T}(f,F,f,1)$, which after several technical considerations is proved to be best possible.Then  by using several calculus arguments the author in \cite{5} provides the evaluation of $B_p^{\mc T}(f,F,L,1)$ for every $L\geq f$. 

Now in \cite{11} the authors give a direct proof of the evaluation of $B_p^{\mc T}(f,F,L,1)$ by using alternative methods. Moreover in the second case, where $L=f$, the author (in \cite{5}) uses the evaluation of $B_p^{\mc T}(f,F,f,1)$ and provides the evaluation of the more general $B_p^{\mc T}(f,F,f,k)$, $k\in (0,1]$.

Our aim (in the future) is to use the results of \cite{Nikol} and of this article in order to approach the following Bellman function problem (of three integral variables)

\begin{multline} \label{eq:0p8}
	B_{p,q}^{\mc T}(f,A,F) = \sup \left\{ \int_X \left(\mc M_{\mc T}\varphi\right)^p\mr \; d\mu: \varphi\geq 0, \int_X\varphi\,\mr \; d\mu=f, \right. \\  \left. \int_X\varphi^q\,\mr \; d\mu = A,\ \int_X\varphi^p\,\mr \; d\mu = F\right\},
\end{multline}
where $1<q<p$, and the variables $f,A,F$ lie in the domain of definition of the above problem.

In \cite{Nikol} we proved that whenever $0<\frac{x^q}{\kappa^{q-1}}<y\leq x^{\frac{p-q}{p-1}}\cdot z^{\frac{q-1}{p-1}}\;\Leftrightarrow\; 0<s_1^{\frac{q-1}{p-1}}\leq s_2<1$, (where $s_1,s_2)$ are defined right below), we find a constant $t=t(s_1,s_2)$ for which if $h:(0,\kappa]\longrightarrow\mb{R}^{+}$ satisfies  $\int_{0}^{\kappa}h=x$\,,\; $\int_{0}^{\kappa}h^q=y$ and $\int_{0}^{\kappa}h^p=z$ then

	$$\int_{0}^{\kappa}\bigg(\frac{1}{t}\int_{0}^{t}h\bigg)^p dt\leq t^p(s_1,s_2)\cdot\int_{0}^{\kappa}h^p,$$
where $t(s_1,s_2)=t$ is the greatest element of $\big[{1,t(0)}\big]$ for which $F_{s_1,s_2}(t)\leq 0$ and $F_{s_1,s_2}$ is defined in \cite{Nikol}. Moreover for each such fixed $s_1,s_2$ 
\[t=t(s_1,s_2)=\min\Big\{{t(\beta)\;:\;\beta\in\big[0,\tfrac{1}{p-1}\big]}\Big\}\]
where $t(\beta)=t(\beta,s_1,s_2)$ (see: \cite{Nikol}). That is we found a constant $t=t(s_1,s_2)$ for which the above inequality  is satisfied for all $h:(0,\kappa]\longrightarrow\mb{R}^{+}$ as mentioned above. Note that $s_1,s_2$ depend by a certain way on $x,y,z$, namely $s_1=\frac{x^p}{\kappa^{p-1}z}$,  $s_2=\frac{x^q}{\kappa^{q-1}y}$ and $F_{s_1,s_2}(\cdot)$ is given in terms of $s_1,s_2$.

In this article we study further the constant $t(s_1,s_2)=t$ which is found in \cite{Nikol}, by providing exact conditions under which the value of $t$ is given by either of the two function branches of its definition (see Section 3).

We need to mention that the extremizers for the standard Bellman function $B_p^{\mc T}(f,F,f,1)$ have been studied in \cite{7}, and in \cite{9} for the case $0<p<1$. Also in \cite{8} the extremal sequences of functions for the respective Hardy operator problem have been studied.  Additionally further study of the dyadic maximal operator can be seen in \cite{10,11} where symmetrization principles for this operator are presented, while other approaches for the determination of certain Bellman functions are given in \cite{13,14,15,16,17}. Moreover results related to applications of dyadic maximal operators can be seen in \cite{12}.

\section{Some technical lemmas}\label{sec:1}

We consider $f,A,F$ variables such that 
\begin{flushright}
$f^q<A<f^{\frac{p-q}{p-1}}\cdot F^{\frac{q-1}{p-1}}\,.\hspace{4.0cm}(*)$
\end{flushright}

We prove

\begin{lemma}\label{lem:1p1}
If $f,A,F$ satisfy the condition $\omega_q\big({\frac{f^q}{A}}\big)>\omega_p\big({\frac{f^p}{F}}\big)$ then there exists unique $\kappa\in\Big({\big({\frac{f^q}{A}}\big)^{\frac{1}{q-1}},1}\Big)$ for which the equality:  $\omega_q\big({\frac{f^q}{\kappa^{q-1}A}}\big)=\omega_p\big({\frac{f^p}{\kappa^{p-1}F}}\big)$, is true.
\end{lemma}
\begin{proof}\label{pro:1p1}
Consider $g:\Big[{\big({\frac{f^p}{F}}\big)^{\frac{1}{p-1}},1}\Big]\longrightarrow\mb{R}^{+}$ defined by 
\[g(\kappa)=\kappa^{q-1}\cdot H_q\Big({\omega_p\big({\tfrac{f^p}{\kappa^{p-1}F}}\big)}\Big)\,.\]
\makebox[\linewidth][s]{This is obviously well defined. We  prove that $g$ is strictly increasing on}\\ $\Big[{\big({\frac{f^p}{F}}\big)^{\frac{1}{p-1}},1}\Big]$\,. We calculate: 
\begin{align}\label{eq:1p1}
g'(\kappa)&=(q-1)\,\kappa^{q-2}\,H_q\big({\omega_p(s(\kappa))}\big)\,+\nonumber\\
&\hspace{3.0cm}+\kappa^{q-1}\,\frac{q(q-1)\,\omega_p(s(\kappa))^{q-2}\big({1-\omega_p(s(\kappa))}\big)}{p(p-1)\,\omega_p(s(\kappa))^{p-2}\big({1-\omega_p(s(\kappa))}\big)}\,s'(\kappa)=\nonumber\\
&=\kappa^{q-2}\,
\bigg[{(q-1)\,H_q\big({\omega_p(s)}\big)-\frac{q(q-1)}{p}\,\frac{1}{\omega_p(s)^{p-q}}\cdot s}\bigg]=\nonumber\\
&=\kappa^{q-2}\,(q-1)\,
\bigg[{H_q\big({\omega_p(s)}\big)-\frac{q}{p}\,\frac{s}{\omega_p(s)^{p-q}}}\bigg]\,,
\end{align}
where $s=s(\kappa)=\frac{f^p}{\kappa^{p-1}F}$\,.\\
We make the change of variables: $\omega_p(s)=t\;\Leftrightarrow\; s=H_p(t)$, where $t\in\big[{1,\frac{p}{p-1}}\big)$\,. Then (\ref{eq:1p1}) gives:
\begin{align}\label{eq:1p2}
g'(\kappa)&=\kappa^{q-2}\,(q-1)\,
\bigg[{H_q(t)-\frac{q}{p}\,\frac{H_p(t)}{t^{p-q}}}\bigg]=\nonumber\\
&=\kappa^{q-2}\,(q-1)\,
\bigg[{\big({q\,t^{q-1}-(q-1)\,t^q}\big)-\frac{q}{p}\,\frac{1}{t^{p-q}}\,\big({p\,t^{p-1}-(p-1)\,t^p}\big)}\bigg]=\nonumber\\
&=\kappa^{q-2}\,(q-1)\,
\bigg[{q\,t^{q-1}-(q-1)\,t^q-\frac{q}{p}\,\big({p\,t^{q-1}-(p-1)\,t^q}\big)}\bigg]=\nonumber\\
&=\kappa^{q-2}\,(q-1)\,
\bigg[{\frac{q}{p}\,(p-1)-(q-1)}\bigg]\,t^q=\nonumber\\
&=\kappa^{q-2}\,(q-1)\,
\bigg[{\frac{qp-q-qp+p}{p}}\bigg]\,t^q=\nonumber\\
&=\frac{(q-1)(p-q)}{p}\,\kappa^{q-2}\,\omega_p(s(\kappa))^{q}\,.
\end{align}
Thus  $g'(\kappa)>0\,,\;\forall\,\kappa\in\Big[{\big({\frac{f^p}{F}}\big)^{\frac{1}{p-1}},1}\Big]$\,.\\
Additionally: $g(1)=H_q\big({\omega_p\big({\tfrac{f^p}{F}}\big)}\big)>\frac{f^q}{A}$\,, by the hypothesis on $f,A,F$ that is stated in Lemma \ref{lem:1p1}.\\
Also  $g\big({\big({\tfrac{f^p}{F}}\big)^{\frac{1}{p-1}}}\big)=\big({\tfrac{f^p}{F}}\big)^{\frac{q-1}{p-1}}<\frac{f^q}{A}$\,, by condition ($\ast$), at the beginning of this section.\\
Thus there exists unique $\kappa\in\Big({\big({\frac{f^q}{A}}\big)^{\frac{1}{q-1}},1}\Big)$ for which 
\[g(\kappa)=\frac{f^q}{A}\;\;\Rightarrow\;\; \frac{f^q}{\kappa^{q-1}A}=H_q\Big({\omega_p\Big({\frac{f^p}{\kappa^{q-1}F}}\Big)}\Big)<1\,.\]
Thus for this $\kappa$, we should also have that $\kappa>\big({\frac{f^q}{A}}\big)^{\frac{1}{q-1}}$\,. Lemma \ref{lem:1p1} now proved.
\end{proof}

Conversely, it is easy to prove the following:
\begin{lemma}\label{lem:1p2}
If $f,A,F$ satisfy ($\ast$), and there exists $\kappa\in\Big({\big({\frac{f^q}{A}}\big)^{\frac{1}{q-1}},1}\Big)$, for which $\omega_q\big({\frac{f^q}{\kappa^{q-1}A}}\big)=\omega_p\big({\frac{f^p}{\kappa^{p-1}F}}\big)$, then $f,A,F$ should satisfy: $\omega_q\big({\frac{f^q}{A}}\big)>\omega_p\big({\frac{f^p}{F}}\big)$\,.
\end{lemma}
\begin{proof}
By the condition on $\kappa$, that is stated in this Lemma, it is implied that $g(\kappa)=\frac{f^q}{A}$\,.\\
Moreover 
since $\kappa<1$ we get $g(\kappa)<g(1)$, by the results of the previous Lemma \ref{lem:1p1}. That is 
\begin{align*}
\frac{f^q}{A}<g(1)&=H_q\bigg({\omega_p\Big({\frac{f^p}{F}}\Big)}\bigg)\quad\Rightarrow\quad
\omega_q\Big({\frac{f^q}{A}}\Big)>\omega_p\Big({\frac{f^p}{F}}\Big)\,.
\end{align*}
\end{proof}
\begin{lemma}\label{lem:1p3}
Suppose that $x,y,z$ are variables satisfying: $\frac{x^q}{\kappa^{q-1}}\leq y\leq x^{\frac{p-q}{p-1}}\cdot z^{\frac{q-1}{p-1}}$\,. Set $s_1=\frac{x^p}{\kappa^{p-1}z}$,  $s_2=\frac{x^q}{\kappa^{q-1}y}$, and assume that $\omega_p(s_1)=\omega_q(s_2):=\gamma$\,. We now set $t:=\omega_p(s_1)$\,. Then $F_{s_1,s_2}(t)=0$\,.
\end{lemma}
\begin{proof}
By definition (see \cite{Nikol}) 
\[F_{s_1,s_2}(t_1)=q\,\big({p\,\omega_q(\tau)^{q-1}-(p-1)\,\omega_q(\tau)^q}\big)\Big({t_1^{p-q}-\frac{s_1}{s_2}}\Big)-(p-q)\,s_1\,\alpha(s_2)\,,\]
where $t_1\in\big[{1,t(0)}\big]$\,, \, $\tau=\tau(t_1)=\dfrac{p-q}{p}\,\dfrac{t_1^p-s_1}{t_1^{p-q}-\frac{s_1}{s_2}}$\,, \, $\tau(t(0))=1$\,, and $\alpha(s_2)=\frac{\omega_q(s_2)^q}{s_2}-1$\,.\\
Initially we prove that $\tau(t)\leq 1$, when $t=\omega_p(s_1)=\gamma$\,. For this purpose we need to prove that
\begin{align}\label{eq:1p3}
\dfrac{p-q}{p}\,\dfrac{\gamma^p-s_1}{\gamma^{p-q}-\frac{s_1}{s_2}}\leq 1\quad&\Leftrightarrow\quad (p-q)(\gamma^p-s_1)\leq p\Big({\gamma^{p-q}-\frac{s_1}{s_2}}\Big)\nonumber\\
&\Leftrightarrow\quad (p-q)\gamma^p-p\,\gamma^{p-q}\leq (p-q)s_1-p\,\frac{s_1}{s_2}\,,
\end{align}
where $\gamma=\omega_p(s_1)$\,. Define the function $\varphi(y)=(p-q)\,y^p-p\,y^{p-q}\,,\; y\geq 1$\,. Then $\varphi(y)$ is strictly increasing on $[1,+\infty)$. We need to prove (\ref{eq:1p3}). By hypothesis $t=\omega_p(s_1)=\omega_q(s_2)\;\Rightarrow\; s_2=H_q\big({\omega_p(s_1)}\big)=H_q(t)$\,. Thus for the proof of (\ref{eq:1p3}) is enough to justify
\begin{align*}
(p-q)\,t^p-p\,t^{p-q}&\leq (p-q)\,H_p(t)-p\,\frac{H_p(t)}{H_q(t)}&\Leftrightarrow\\
H_q(t)\,\big[{(p-q)\,t^p-p\,t^{p-q}}\big]&\leq(p-q)\,H_p(t)\cdot H_q(t)-p\,H_p(t)\,.
\end{align*}
The last stated inequality is easily seen to be equivalent to the following:
\[G(t):=(2pq-p)\,t^q-(q-1)p\,t^{q+1}-pq\,t^{q-1}-p\,t+p\leq0\]
which is true since $t\geq1$\,,\; $G(1)=0$, and 
\begin{align*}
G(t)=-p(t-1)+pq\,t^q-pq\,t^{q+1}+p\,t^{q}(t-1)+pq\,t^{q-1}(t-1)\,.
\end{align*}
Then if $t=1\;\Rightarrow\;G(t)=0$\,, while if $t>1$ we get that 
\begin{align*}
G(t)&=-p(t-1)+pq\,t^q-pq\,t^{q+1}+p\,t^{q}(t-1)+pq\,t^{q-1}(t-1)\\
    &\cong -p+pq(-t^q)+p\,t^q+pq\,t^{q-1}\\
    &\cong -1-q\,t^q+t^q+q\,t^{q-1}\\
    &=q\,t^{q-1}-(q-1)\,t^q-1\\
    &=H_q(t)-1<0\,,
\end{align*}
since $t>1$\,.\\
As we have already seen, when $s_1<1\;\Leftrightarrow\; t=\omega_p(s_1)>1$ we have that $\tau(t)<1$ (this is a consequence of the proof that is already given in \cite{Nikol}).\\
Thus $\tau(t)<1=\tau(t(0))\;\Rightarrow\; 1<t=\omega_p(s_1)=\gamma<t(0)$.\\
We finally prove that 
\begin{align}\label{eq:1p4}
F_{s_1,s_2}(t)&=0\quad\Leftrightarrow\nonumber\\
q\,\big({p\,\omega_q(\tau)^{q-1}-(p-1)\,\omega_{q}(\tau)^q}\big)\Big({t^{p-q}-\frac{s_1}{s_2}}\Big)&=(p-q)\,s_1\cdot\alpha(s_2)
\end{align}
Note that 
\begin{align*}
&p\,\omega_q(\tau)^{q-1}-(p-1)\,\omega_q(\tau)^q=\\
&q\,\omega_q(\tau)^{q-1}+(p-q)\,\omega_q(\tau)^{q-1}-(q-1)\,\omega_q(\tau)^q-(p-q)\,\omega_q(\tau)^q=\\
&\big[{q\,\omega_q(\tau)^{q-1}-(q-1)\,\omega_q(\tau)^q}\big]+(p-q)\,\omega_q(\tau)^{q-1}\big({1-\omega_q(\tau)}\big)=\\
&H_q\big({\omega_q(\tau)}\big)+(p-q)\,\omega_q(\tau)^{q-1}\big({1-\omega_q(\tau)}\big)=\\
&\tau+(p-q)\,\omega_q(\tau)^{q-1}\big({1-\omega_q(\tau)}\big)\,.
\end{align*}
Thus for the proof of (\ref{eq:1p4}) it's enough to show that:
\begin{align}\label{eq:1p5}
q(t^p-s_1)+pq\,\big({\omega_q(\tau)^{q-1}-\omega_q(\tau)^q}\big)\Big({t^{p-q}-\frac{s_1}{s_2}}\Big)=p\,s_1\cdot\alpha(s_2)
\end{align}
(\ref{eq:1p5}) is then equivalent to:
\begin{align}\label{eq:1p6}
q(t^p-s_1)+p\,\big({\tau-\omega_q(\tau)^q}\big)\Big({t^{p-q}-\frac{s_1}{s_2}}\Big)&=p\,s_1\cdot\alpha(s_2)&\Leftrightarrow\footnotemark\nonumber\\
(t^p-s_1)-\Big({t^{p-q}-\frac{s_1}{s_2}}\Big)\,\omega_q(\tau)^q&=s_1\cdot\alpha(s_2)&\Leftrightarrow\nonumber\\
\frac{p}{p-q}\,\tau-\omega_q(\tau)^q&=\frac{s_1\big({\frac{t^q}{s_2}-1}\big)}{t^{p-q}-\frac{s_1}{s_2}}&\Leftrightarrow\nonumber\\
\omega_q(\tau)^q&=\frac{t^{p}-\frac{s_1}{s_2}\,t^q}{t^{p-q}-\frac{s_1}{s_2}}&\Leftrightarrow\nonumber\\
\omega_q(\tau)&=t\,.
\end{align}
\footnotetext{by definition of $\tau$.}
We prove (\ref{eq:1p6}). We have: (\ref{eq:1p6})\quad$\Leftrightarrow\quad\tau=H_q(t)\quad\Leftrightarrow$
\begin{align}\label{eq:1p7}
\frac{p-q}{p}\,\frac{t^{p}-s_1}{t^{p-q}-\frac{s_1}{s_2}}=s_2\quad\Leftrightarrow\quad p\,s_2\,t^{p-q}-(p-q)\,t^{p}=q\,s_1\,.
\end{align}
But 
\begin{align}\label{eq:1p8}
&\frac{H_p(t)}{H_q(t)}=\frac{s_1}{s_2}\quad\Rightarrow\quad\frac{p\,t^{p-1}-(p-1)\,t^{p}}{q\,t^{q-1}-(q-1)\,t^{q}}=\frac{s_1}{s_2}\quad\Rightarrow\nonumber\\
&p\,s_2\,t^{p-q}=(p-1)s_2\,t^{p-q+1}+s_1\big({q-(q-1)t}\big)\,.
\end{align}
Thus in view of (\ref{eq:1p8}), (\ref{eq:1p7}) is equivalent to 
\begin{align*}
&(p-q)\,t^{p}=(p-1)s_2\,t^{p-q+1}-s_1(q-1)t&\Leftrightarrow\\
&(p-q)\,t^{p-1}=(p-1)s_2\,t^{p-q}-(q-1)s_1&\Leftrightarrow\\
&(p-q)\,t^{p-1}-(p-1)\,H_q(t)\,t^{p-q}+(q-1)\,H_p(t)=0\,,
\end{align*}
which is true in view of the definition of $H_p(t)$ and $H_q(t)$. Lemma \ref{lem:1p3} is now proved.
\end{proof}

As a consequence of Lemma \ref{lem:1p3} we obtain that if 
\begin{align}\label{eq:1p9}
\omega_p(s_1)=\omega_q(s_2)=\gamma>1\,,
\end{align}
then $F_{s_1,s_2}\big({\omega_p(s_1)}\big)=0$\,, with $1<\gamma=\omega_p(s_1)<t(0)$\,. Thus (see: \cite{Nikol}) we must have (in case where (\ref{eq:1p9}) is true) that $t'(0)<0$\,. Indeed if  $t'(0)\geq0$ (see the comments in the end of Section 4  in \cite{Nikol}) we should have $F_{s_1,s_2}\big({t(0)}\big)\leq0$. But we already have that $F_{s_1,s_2}(\gamma)=0$\,, where $\gamma<t(0)\;\Rightarrow\; F_{s_1,s_2}(\gamma)<F_{s_1,s_2}\big({t(0)}\big)\;\Rightarrow\; F_{s_1,s_2}\big({t(0)}\big)>0$\,, which is contradiction.\\
Thus (see \cite{Nikol}), we have that there exists $\beta\in \big({0,\frac{1}{p-1}}\big)$ :\; 
$t'(\beta)=0$ for which
\begin{align}\label{eq:1p10}
t(\beta)=\min\Big\{{t(\gamma_1)\;:\;\gamma_1\in\big[0,\tfrac{1}{p-1}\big]}\Big\}\,.	
\end{align}
But for this $\beta$, we have: \; $1\leq t(\beta)<t(0)$ and $F_{s_1,s_2}\big({t(\beta)}\big)=0$ and since $\gamma=\omega_p(s_1)$ also satisfies  $F_{s_1,s_2}(\gamma)=0$, we obtain
\[t(\beta)=\omega_p(s_1)\,.\]
We state the above comments as
\begin{corollary}\label{cor:1p1}
If the variables $s_1,s_2$ satisfy $\omega_p(s_1)=\omega_q(s_2)$ then there exists  $\beta\in \big({0,\frac{1}{p-1}}\big)$ :\; 
$t'(\beta)=0$ for which
\begin{align*}
	t(\beta)=\min\Big\{{t(\gamma_1)\;:\;\gamma_1\in\big[0,\tfrac{1}{p-1}\big]}\Big\}=\omega_p(s_1)\,.	
\end{align*}
\end{corollary}

\begin{lemma}\label{lem:1p4}
For each $\ell\in(-\infty,0]$ the following equation is true:
\begin{align*}
\lim\limits_{\alpha\to0^{+}}\alpha\cdot\omega_p\Big({\frac{\ell}{\alpha}}\Big)^p=-\frac{1}{p-1}\cdot\ell\,.
\end{align*}
\end{lemma}
\begin{proof}
We have
\begin{align}\label{eq:1p11}
L:=\lim\limits_{\alpha\to0^{+}}\alpha\cdot\omega_p\Big({\frac{\ell}{\alpha}}\Big)^p=\lim\limits_{s\to+\infty}\frac{\omega_p(\ell\cdot s)^p}{s}\,.
\end{align}
Assume that $\ell<0$, then as $s\to+\infty$, $\ell s\to-\infty$\,, that $\omega_p(\ell s)\to+\infty$. Thus by (\ref{eq:1p11})
\begin{align*}
	L=\lim\limits_{s\to+\infty}\bigg[{p\,\omega_p(\ell s)^{p-1}\frac{d}{ds}\big({\omega_p(\ell s)}\big)}\bigg]\,.
\end{align*}
We compute 
\begin{align*}
\frac{d}{ds}\omega_p(\ell s)&=\ell\,\omega'_p(\ell s)=\ell\,(H^{-1}_p)'(\ell s)=\\
&=\ell\,\frac{1}{p(p-1)\,\omega_p(\ell s)^{p-2}\big({1-\omega_p(\ell s)}\big)}\,.
\end{align*}
Thus
\begin{align*}
	L=\lim\limits_{s\to+\infty}\bigg[{\frac{\ell}{p-1}\,\frac{\omega_p(\ell s)}{1-\omega_p(\ell s)}}\bigg]=-\frac{1}{p-1}\,\ell\,.
\end{align*}
Note also that for $\ell=0$, obviously $L=0$\,.
\end{proof}
We now prove the following
\begin{lemma}\label{lem:1p5}
If $1<q<p$ and $\lambda\in[0,1)$ the inequality 
\[\omega_p(\lambda^{p-1})<\omega_q(\lambda^{q-1})\,,\]
is true.
\end{lemma}
\begin{proof}
The inequality stated in this lemma is equivalent to 
\begin{align}\label{eq:1p12}
H_q\big({\omega_p(\lambda^{p-1})}\big)&>\lambda^{q-1}\,,\quad\forall\,\lambda\in[0,1)\quad\Leftrightarrow\nonumber\\
q\,{\omega_p(\lambda^{p-1})}^{q-1}-(q-1)\,{\omega_p(\lambda^{p-1})}^{q}&>\lambda^{q-1}\,,\quad\forall\,\lambda\in[0,1)\,.
\end{align}
We set $\lambda^{p-1}=t\in[0,1)$ in (\ref{eq:1p12}) and thus it is sufficient to prove
\begin{align}\label{eq:1p13}
q\,{\omega_p(t)}^{q-1}-(q-1)\,{\omega_p(t)}^{q}-t^{\frac{q-1}{p-1}}>0\,,\quad\forall\,t\in[0,1)\,.
\end{align}
Denote by $G_1(t)$ the left side of (\ref{eq:1p13}), for each $t\in[0,1)$\,. Then 
\begin{align*}
G'_1(t)&=\big[{H_q\big({\omega_p(t)}\big)}\big]'-\frac{q-1}{p-1}\,\frac{1}{t^{\frac{p-q}{p-1}}}=\\
 &=\frac{q(q-1)\,\omega_p(t)^{q-2}\big({1-\omega_p(t)}\big)}{p(p-1)\,\omega_p(t)^{p-2}\big({1-\omega_p(t)}\big)}-\frac{q-1}{p-1}\,\frac{1}{t^{\frac{p-q}{p-1}}}\cong\\
 &\cong \frac{q}{p}\,\frac{1}{\omega_p(t)^{p-q}}-\frac{1}{t^{\frac{p-q}{p-1}}}<\frac{q}{p}-1<0\,,\quad\forall\,t\in(0,1)\,.
\end{align*}
Thus $G_1$ is strictly decreasing function of $t\in(0,1)$, thus is satisfies $G_1(t)>G_1(1)\,,\;\forall\,t\in(0,1)$, while $G_1(1)=0$, so that $G_1(t)\geq0\,,\;\forall\,t\in[0,1)$ and the Lemma is proved.
\end{proof}

\section{Properties of the constant $t(s_1,s_2)$}

We set $D=\big\{{(s_1,s_2)\in\mb{R}^2\;:\; 0<s_1^{q-1}\leq s_2^{p-1}<1}\big\}\subset\mb{R}^2$\,.\\
As we have seen in \cite{Nikol} we have associated, to any $(s_1,s_2)\in D$  and any $\beta\in \big[{0,\frac{1}{p-1}}\big]$, a real number $t(s_1,s_2,\beta)=t_{s_1,s_2}(\beta)=t(\beta)$ according to the results of \cite{Nikol}. If $t'(0)<0$, we define by $t=t_{s_1,s_2}=t(s_1,s_2)$ the number $t=\min\Big\{{t(\gamma)\;:\;\gamma\in\big[0,\tfrac{1}{p-1}\big]}\Big\}$ which satisfies the relation $F_{s_1,s_2}(t)=0$ (see \cite{Nikol}). Thus we have an expression for $t$ whenever $t'(0)<0$. \\
Now we are interested to find exactly those $(s_1,s_2)\in D$ for which  $t'(0)<0$. For this purpose, in a first step, we find a condition that guarantees us the inequality   $t'(0)<0$.\\
Remember that $t(s_1,s_2,0)=t_{s_1,s_2}(0)=t(0)$ is defined by the relation 
\[t^{p}(0)-\frac{p}{p-q}\,t^{p-q}(0)=s_1-\frac{p}{p-q}\,\frac{s_1}{s_2}=:h(s_1,s_2)\]
for each $(s_1,s_2)\in D$, where $t(0)\geq 1$\,.\\
Remember that the function $\varphi:(0,+\infty)\longrightarrow\mb{R}$, defined by $\varphi(y)=y^p-\frac{p}{p-q}\,y^{p-q}$\,, is strictly decreasing on $(0,1)$,\; is strictly increasing on $(1,+\infty)$ and attain it's minimum value at $y_0=1$, which is $\varphi(1)=-\frac{q}{p-q}$\,.

\pagebreak

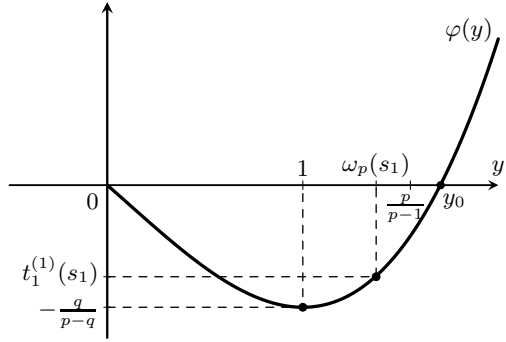
\begin{minipage}[h]{0.5\textwidth}
\centering
\begin{tikzpicture}[line cap=round,line join=round,>=stealth,x=2cm,y=1.25cm,scale=0.65]
\draw[->,color=black,line width=0.7pt] (-1,0.) -- (4,0.);
\foreach \x in {2,2.75,3.1}
\draw[shift={(\x,0)},color=black] (0pt,-2.5pt) -- (0pt,2.5pt)
node[below] at (0,-1pt) {};	
\draw[color=black] (4,0) node[above]  {{\small{$y$}}};
	\draw[->,color=black,line width=0.89pt] (0,-2.5) -- (0.,3.);
	\foreach \y in {-2,-1.5}
	\draw[shift={(0,\y)},color=black] (-2pt,0pt) -- (2pt,0pt) node[left] at (-1pt,0pt) {};
	
	\draw[shift={(0.,0.)},line width=1.25pt,color=black] plot[smooth,samples=200,domain=0.0:4.,variable=\t] ({\t},{0.15*(\t-2)^3+0.8*(\t-2)^2-2});
\draw[-,color=black,line width=0.5pt,dashed] (0,-2) -- (2,-2) -- (2,0);
\draw[-,color=black,line width=0.5pt,dashed] (0,-1.5) -- (2.75,-1.5) -- (2.75,0);
	\draw [fill=black] (2,-2) circle (2.25pt);
	\draw [fill=black] (2.75,-1.5) circle (2.25pt);
	\draw [fill=black] (3.41,0) circle (2.1pt);
	\begin{small}
		\draw[color=black] (-0.0,-1.4) node[left] {$t^{(1)}_1(s_1)$};
		\draw[color=black] (-0.0,-2.1) node[left] {$-\frac{q}{p-q}$};
		\draw[color=black] (-0.15,0) node[below] {$0$};
		\draw[color=black] (3.55,0) node[below]  {$y_0$};
		\draw[color=black] (2,0) node[above]  {$1$};
		\draw[color=black] (2.75,0) node[above]  {$\omega_p(s_1)$};
		\draw[color=black] (3.05,0) node[below]  {$\frac{p}{p-1}$};
		\draw[color=black] (3.7,2.2) node[above]  {$\varphi(y)$};
	\end{small}
\end{tikzpicture}
\end{minipage}
\hspace{1.45cm}\begin{minipage}[h]{0.33\textwidth}
If $y_0>1$ is such that $\varphi(y_0)=0$ then we easily see that $y_0$ should satisfy:\\
$y_0=\big({\frac{p}{p-q}}\big)^{1/q}>\frac{p}{p-1}>1$.
\end{minipage}\vspace{0.3cm}\\
Also the function $h:D\longrightarrow\mb{R}$, defined by $h(s_1,s_2)=s_1-\frac{p}{p-q}\,\frac{s_1}{s_2}$ satisfies $h(s_1,s_2)<0$,\,$\forall\,(s_1,s_2)\in D$, while we also have $h(s_1,s_2)>-\frac{q}{p-q}$,\,$\forall\,(s_1,s_2)\in D$. This is true since $h(s_1,s_2)\geq s_1-\frac{p}{p-q}\,s_1^{\frac{p-q}{p-1}}=:g(s_1)$\,, where $g'(s_1)=1-\frac{p}{p-1}\,\frac{1}{s_1^{\frac{q-1}{p-1}}}<0$\,,\,$\forall\,s_1\in (0,1)$. Thus $g(s_1)>g(1)=-\frac{q}{p-q}$\,,\,$\forall\,s_1\in (0,1)$.\\
Then since $s_1\in (0,1)$, and $\varphi\big|_{[1,+\infty)}$ is strictly increasing, and also $1<\omega_p(s_1)<\frac{p}{p-q}<y_0=\big({\frac{p}{p-q}}\big)^{1/q}$, we obtain that there exists (for every $s_1\in (0,1)$) a real number $t^{(1)}_1(s_1)\in\Big({-\frac{q}{p-q},0}\Big)$ for which 
\[\varphi\big({\omega_p(s_1)}\big)=t^{(1)}_1(s_1)\quad\Leftrightarrow\quad \omega_p(s_1)=\varphi^{-1}\big({t^{(1)}_1(s_1)}\big)\,.\]
If now $(s_1,s_2)\in D$ is such that $h(s_1,s_2)>t^{(1)}_1(s_1)$ then we get 
\begin{align}\label{eq:2p1}
\varphi\big({t(0)}\big)>\varphi\big({\omega_p(s_1)}\big)\quad\Rightarrow\quad t(0)>\omega_p(s_1)\,.
\end{align}
Condition (\ref{eq:2p1}) is the one we search for, as a first step, in order to find those $(s_1,s_2)\in D$ for which $t'(0)<0$. Note that if $t(0)>\omega_p(s_1)$ then $t'(0)<0$ (the converse is not true as we shall see further below in this section). For the proof of the above claim, note that if we denote by $\beta_0=\omega_p(s_1)-1$, then $\beta_0\in\big({0,\frac{1}{p-1}}\big)$ and $t(\beta_0)\leq t_1(\beta_0)=\omega_p(s_1)$ (see Section 3 in \cite{Nikol}).\\
Thus we get $t(\beta_0)\leq \omega_p(s_1)<t(0)\;\;\Rightarrow\;\; \exists\,\beta_1\in\big({0,\frac{1}{p-1}}\big)$, for which $t'(\beta_1)=0$, \, $t(\beta_1)<t(0)$ (see \cite{Nikol}). Thus we must have that $t'(0)<0$, otherwise as we show in \cite{Nikol} we should have that 
\[t(0)=\min\Big\{{t(\gamma)\;:\;\gamma\in\big[0,\tfrac{1}{p-1}\big]}\Big\}\,.\]
We have just proved that the condition $\omega_p(s_1)<t(0)$ implies $t'(0)<0$. Fix now $s_1\in (0,1)$, and denote by $h_{s_1}$ the function  
\[h_{s_1}(s_2)=h(s_1,s_2)=s_1-\frac{p}{p-q}\,\frac{s_1}{s_2}\,,\]
where $s_2\in \Big[{s_1^{\frac{q-1}{p-1}},1}\Big]$. Then $h_{s_1}$ is strictly increasing function of $s_2$, and 
\[h_{s_1}\Big({\Big[{s_1^{\frac{q-1}{p-1}},1}\Big]}\Big)=\Big[{g(s_1):=s_1-\tfrac{p}{p-q}\,s_1^{\frac{p-q}{p-1}},-\tfrac{q}{p-q}\,s_1}\Big]\,.\]
Thus if  $\varphi\big({\omega_p(s_1)}\big):=t^{(1)}_1(s_1)<g(s_1)$ we should have that \[\varphi\big({\omega_p(s_1)}\big)<g(s_1)\leq h(s_1,s_2)=\varphi\big({t(0)}\big)\quad\Rightarrow\quad  \omega_p(s_1)<t(0)\,,\]
which is condition (\ref{eq:2p1}). Now fix $s_1\in (0,1)$. We search for those $s_1$'s for which 
\begin{align}\label{eq:2p2}
\varphi\big({\omega_p(s_1)}\big)<g(s_1)
\end{align} is true.
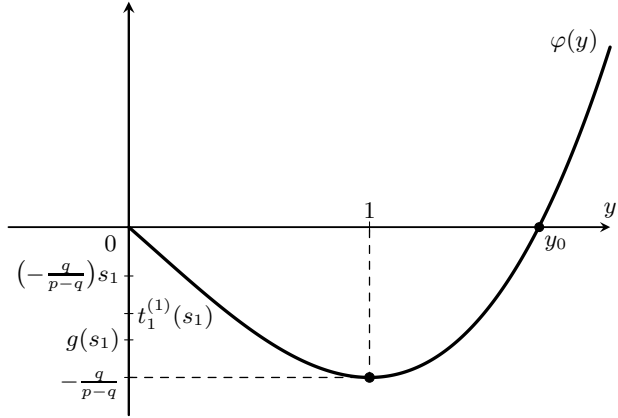
\begin{center}
\begin{tikzpicture}[line cap=round,line join=round,>=stealth,x=2cm,y=1.25cm,scale=0.8]
	\draw[->,color=black,line width=0.7pt] (-1,0.) -- (4,0.);
	\foreach \x in {2}
	\draw[shift={(\x,0)},color=black] (0pt,-2.5pt) -- (0pt,2.5pt)
	node[below] at (0,-1pt) {};	
	\draw[color=black] (4,0) node[above]  {{\small{$y$}}};
	\draw[->,color=black,line width=0.89pt] (0,-2.5) -- (0.,3.);
	\foreach \y in {-2,-1.5,-1.15,-0.65}
	\draw[shift={(0,\y)},color=black] (-2pt,0pt) -- (2pt,0pt) node[left] at (-1pt,0pt) {};
	
	\draw[shift={(0.,0.)},line width=1.25pt,color=black] plot[smooth,samples=200,domain=0.0:4.,variable=\t] ({\t},{0.15*(\t-2)^3+0.8*(\t-2)^2-2});
	\draw[-,color=black,line width=0.5pt,dashed] (0,-2) -- (2,-2) -- (2,0);
	\draw [fill=black] (2,-2) circle (2.25pt);
	\draw [fill=black] (3.41,0) circle (2.1pt);
	\begin{small}
		\draw[color=black] (-0.0,-1.15) node[right] {$t^{(1)}_1(s_1)$};
		\draw[color=black] (-0.0,-2.1) node[left] {$-\frac{q}{p-q}$};
		\draw[color=black] (-0.15,0) node[below] {$0$};
		\draw[color=black] (3.55,0) node[below]  {$y_0$};
		\draw[color=black] (2,0) node[above]  {$1$};
		\draw[color=black] (0,-0.65) node[left]  {$\big({-\frac{q}{p-q}}\big)s_1$};
		\draw[color=black] (0,-1.5) node[left]  {$g(s_1)$};
		\draw[color=black] (3.7,2.2) node[above]  {$\varphi(y)$};
	\end{small}
\end{tikzpicture}
\end{center}
(\ref{eq:2p2}) is equivalent to 
\begin{align*}
&{\omega_p(s_1)}^{p}-\frac{p}{p-q}\,{\omega_p(s_1)}^{p-q}<s_1-\frac{p}{p-q}\,s_1^{\frac{p-q}{p-1}}\qquad\Leftrightarrow\\
\vartheta(s_1):=\,&{\omega_p(s_1)}^{p}-\frac{p}{p-q}\,{\omega_p(s_1)}^{p-q}-s_1+\frac{p}{p-q}\,s_1^{\frac{p-q}{p-1}}<0\,.
\end{align*}
Note that $\vartheta(1)=0$, while 
\begin{align}\label{eq:2p3}
\vartheta'(s_1)&=p\,{\omega_p(s_1)}^{p-1}\omega'_p(s_1)-p\,{\omega_p(s_1)}^{p-q-1}\omega'_p(s_1)-1+\frac{p}{p-1}\,s_1^{-\frac{q-1}{p-1}}=\nonumber\\
&=p\,{\omega_p(s_1)}^{p-q-1}\omega'_p(s_1)\big({{\omega_p(s_1)}^{q}-1}\big)+\frac{p}{p-1}\,s_1^{-\frac{q-1}{p-1}}-1\qquad\Rightarrow\nonumber\\
\vartheta'(s_1)&=\Delta(s_1)+\frac{p}{p-1}\,s_1^{-\frac{q-1}{p-1}}-1\,,
\end{align}
where $\Delta(s_1)$ is defined by (\ref{eq:2p3}). Then 
\begin{align*}
\Delta(s_1)&=p\,\frac{1}{p(p-1)\,{\omega_p(s_1)}^{p-2}\big({1-{\omega_p(s_1)}}\big)}\cdot{\omega_p(s_1)}^{p-q-1}\big({{\omega_p(s_1)}^{q}-1}\big)=\\
 &=-\frac{1}{p-1}\,{\omega_p(s_1)}^{-q+1}\cdot\frac{{\omega_p(s_1)}^{q}-1}{\omega_p(s_1)-1}\,.
\end{align*}

\pagebreak

\noindent Then, 
\begin{align*}
\Delta(1^{-})=-\frac{1}{p-1}\lim\limits_{t\to1^{+}}\frac{t^q-1}{t-1}=-\frac{q}{p-1}\,,
\end{align*}
thus
\begin{align*}
\vartheta'(1^{-})=\Delta(1^{-})+\frac{p}{p-1}-1=\frac{q}{p-1}+\frac{p}{p-1}-1=\frac{p-q}{p-1}-1<0\,.
\end{align*}
Until now we have proved that $\vartheta'(1^{-})<0$\,. Thus, since $\vartheta(1)=0$, (\ref{eq:2p2}) cannot be true if we let $s_1$ arbitrarily close to $1^{-}$. Also 
\begin{align*}
\vartheta(0^{+})=\vartheta(0)=\bigg({\frac{p}{p-1}}\bigg)^{p-q}\bigg[{\bigg({\frac{p}{p-1}}\bigg)^{q}-\frac{p}{p-q}}\bigg]<0\,.
\end{align*}
Thus (\ref{eq:2p2}) holds true if we let $s_1\to 0^{+}$.

We continue to search for which $s_1\in (0,1)$ we have that inequality  (\ref{eq:2p2}) : that is $\varphi\big({\omega_p(s_1)}\big)<g(s_1)$ is true. We prove that the function $\vartheta(s_1)$ is strictly concave on $s_1\in (0,1)$.\vspace{0.2cm}\\
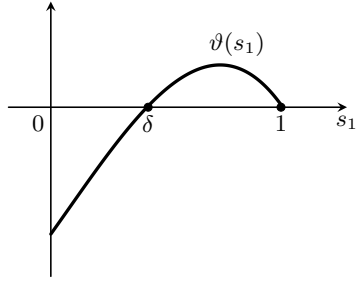
\begin{minipage}[h]{0.45\textwidth}
	\centering
	\begin{tikzpicture}[line cap=round,line join=round,>=stealth,x=1.5cm,y=1.5cm,scale=0.75]
		\draw[->,color=black,line width=0.6pt] (-0.5,0.) -- (3.5,0.);
		\foreach \x in {}
		\draw[shift={(\x,0)},color=black] (0pt,-2.5pt) -- (0pt,2.5pt)
		node[below] at (0,-1pt) {};	
		\draw[->,color=black,line width=0.6pt] (0,-2.) -- (0.,1.25);
		\foreach \y in {}
		\draw[shift={(0,\y)},color=black] (-2pt,0pt) -- (2pt,0pt) node[left] at (-1pt,0pt) {};
		\draw[shift={(0.,0.)},line width=1.25pt,color=black] plot[smooth,samples=200,domain=0.0:2.75,variable=\t] ({\t},{-0.15*(\t-2)^3-0.8*(\t-2)^2+0.5});
		\draw [fill=black] (1.15,0) circle (2.pt);
		\draw [fill=black] (2.72,0) circle (2.pt);
		\begin{small}
			\draw[color=black] (1.15,0) node[below] {$\delta$};
			\draw[color=black] (2.2,0.5) node[above] {$\vartheta(s_1)$};
			\draw[color=black] (-0.15,0) node[below] {$0$};
			\draw[color=black] (3.5,0) node[below]  {$s_1$};
			\draw[color=black] (2.72,0) node[below]  {$1$};
			\end{small}
	\end{tikzpicture}
\end{minipage}
\hspace{0.55cm}\begin{minipage}[h]{0.5\textwidth}
If we prove this fact then by the above reasoning we should have that there exists $\delta\in (0.1)$ such that
\begin{enumerate}[label={\roman*)}]
\item
$\forall\,s_1\in(0,\delta):\; \vartheta(s_1)<0$\,.
\item
$\forall\,s_1\in(\delta,1): \vartheta(s_1)>0$, with  $\vartheta(\delta)=0$\,.
\end{enumerate}
\end{minipage}\vspace{0.3cm}\\
Thus when $s_1\in(0,\delta)$ we would have for each $s_2\in \Big[{s_1^{\frac{q-1}{p-1}},1}\Big)$ that $\omega_p(s_1)<t(0)$ as we noted above and thus  $t'(0)<0$. We have (by (\ref{eq:2p3})):
\begin{align}\label{eq:2p4}
	\vartheta'(s_1)&=\Delta(s_1)+\frac{p}{p-1}\,s_1^{-\frac{q-1}{p-1}}-1=\nonumber\\
	&=-\frac{1}{p-1}\,\frac{\omega_p(s_1)-{\omega_p(s_1)}^{-q+1}}{\omega_p(s_1)-1}+\frac{p}{p-1}\,s_1^{-\frac{q-1}{p-1}}-1\qquad\Rightarrow\nonumber\\
	\vartheta'(s_1)&=-\frac{1}{p-1}\,\lambda\big({\omega_p(s_1)}\big)+\frac{p}{p-1}\,s_1^{-\frac{q-1}{p-1}}-1\,,
\end{align}
where $\lambda(x)=\dfrac{x-x^{-q+1}}{x-1}\,,\; x>1$\,.\\
Then we calculate  the sign of $\lambda'(x)$, as follows:
\begin{align}\label{eq:2p5}
\lambda'(x)&\cong\big({1-({-q+1})x^{-q}}\big)(x-1)-(x-x^{-q+1})=\nonumber\\
           &=x-1+(q-1)x^{-q+1}-(q-1)x^{-q}-x+x^{-q+1}=\nonumber\\
           &=qx^{-q+1}-(q-1)x^{-q}-1=:\epsilon(x)\hspace{2.5cm}\Rightarrow\nonumber\\           
\end{align}
Then $\epsilon'(x)<0$\,,\, $\forall\,x>1$, thus $\epsilon(x)<\epsilon(1)=0$\,,\, $\forall\,x>1$. 

\pagebreak

\noindent\hspace*{1.275cm} Thus $\lambda'(x)<0$\,,\, $\forall\,x>1$, \\
$=\!=\!\Longrightarrow\quad\lambda(x)$\; is strictly decreasing on $x\in(1,+\infty)$\\
$=\!=\!\Longrightarrow\quad\lambda\big({\omega_p(s_1)}\big)$\; is strictly increasing of  $s_1\in (0,1)$\\
$\stackrel{\text{by}\,(\ref{eq:2p4})}{=\!=\!\Longrightarrow}\quad\vartheta'(s_1)$\; is strictly decreasing\\
$=\!=\!\Longrightarrow\quad\vartheta(s_1)$ is strictly concave on the interval: $s_1\in (0,1)$.\\
This $\delta$ (as described before) is determined by the equation
\begin{align}\label{eq:2p6}
\vartheta(\delta)&=0 &\Leftrightarrow\nonumber\\
\omega_p(\delta)^{p}-\frac{p}{p-q}\,\omega_p(\delta)^{p-q}&=\delta-\frac{p}{p-q}\,\delta^{\frac{p-q}{p-1}}\,,
\end{align}
with $s_1\in (0,1)$.\\
Note that $\delta=\delta(p,q)\in(0,1)$ depends only on $p,q$. We remind again now that if $s_1\in (0,\delta)$ then $\vartheta(s_1)<0\;\Rightarrow\; \omega_p(s_1)<t_{s_1,s_2}(0)$\,,\,$\forall\,s_2\in \Big[{s_1^{\frac{q-1}{p-1}},1}\Big)$, thus $t'_{s_1,s_2}(0)<0$, for every such $s_1$ and $s_2$.\\
Moreover we need to study the case where $\delta\leq s_1<1$. If $s_1\in (\delta,1)$ then $\varphi\big({\omega_p(s_1)}\big)=t^{(1)}_1(s_1)>g(s_1)$ because $s_1\in (\delta,1)$ implies $\vartheta(s_1)>0$ and thus $\varphi\big({\omega_p(s_1)}\big)>g(s_1)$. Note now that for each $s_1\in (\delta,1)$
\begin{align}\label{eq:2p7}
    g(s_1)<t^{(1)}_1(s_1)<-\frac{q}{p-q}\,s_1\,.
\end{align}
We just need to prove the second inequality in (\ref{eq:2p7}). We prove this, by using the following reasoning:
\begin{align} \label{eq:2p8}
    &t_1^{(1)} < - \frac{q}{p-q} s_1 \Leftrightarrow \phi(\omega_{p}(s_1)) < - \frac{q}{p-q} s_1 \Leftrightarrow \nonumber \\
    &G(s):= \omega_p(s_1)^p - \frac{p}{p-q} \omega_p(s_1)^{p-q} + \frac{q}{p-q} s_1 < 0 , \forall s_1 \in (\delta,1) .
\end{align}
As before, we calculate
\[
    G'(s_1) = \Delta(s_1) + \frac{q}{p-q}, \quad s_1 \in (\delta,1),
\]
where $\Delta(s_1)$ is given by \eqref{eq:2p3}. Moreover
\begin{align*} 
    G'(1^-) &= \Delta(1^-) + \frac{q}{p-q} = -\frac{q}{p-1} + \frac{q}{p-q} \\
    &= q \left ( \frac{1}{p-q} - \frac{1}{p-1} \right ) = q  \left ( \frac{p-1 - (p-q)}{(p-1)(p-q)} \right )  \\
    &= \frac{q(q-1)}{(p-1)(p-q)} > 0,
\end{align*}
while $\Delta(s_1)$ is strictly decreasing on $s_1 \in (0,1)$, thus $G$ is strictly concave on $s_1 \in (0,1)$ with $G(1)=0$ and $G'(1^-)>0.$ These observations imply that $G(s_1)<0 , \forall s_1 \in (\delta,1)$, thus \eqref{eq:2p8} holds true.

\pagebreak

\noindent Thus $t^{1}({s_1})$ lives on the interval $ \left (g(s_1), -\frac{q}{p-q}s_1 \right )$, thus also in the range \\ $h_{s_1} \left ( s_1^{(q-1)/(p-1)},1 \right )$ and by the strict monotonicity of $h_{s_1}$, we get that there exists unique
\[
s_2' = s_2'(s_1) \in  \left ( s_1^{(q-1)/(p-1)},1 \right ),
\]
for which:
\[
    h_{s_1}(s_2') = h(s_1, s_2') = t_1^{(1)}(s_1) .
\]
Now we consider two cases
\begin{itemize}
    \item \( a ) \) $s_2 \in (s_2'(s_1),1), $ that is $s_2'<s_2<1.$ Then:
    $$ \phi(\omega_p(s_1)) = t_1^{(1)}(s_1) = h(s_1,s_2') \overset{s_2'<s_2}{<} h(s_1,s_2) = \phi(t(0)) ,$$
    which implies $\omega_p(s_1) < t(0),$ that is condition \eqref{eq:2p1}. 

    \item \( b ) \) $s_1^{(q-1)/(p-1)} < s_2 \leq s_2' = s_2'(s_1). $ Then:
    \begin{align*}
        g(s_1) &= h(s_1, s_1^{(q-1)/(p-1)}) = h_{s_1}(s_1^{(q-1)/(p-1)} ) < h_{s_1}(s_2) \\
    &\leq h(s_1,s_2') = t_1^{(1)}(s_1).
    \end{align*}

    Also in this case, $s_2 \leq s_2' \Rightarrow$
    \begin{align*}
        &\phi(t(0)) = h(s_1,s_2) \leq h(s_1, s_2') = t_1^{(1)}(s_1) = \phi(\omega_p(s_1)) \\
        &\Rightarrow t(0) \leq \omega_p(s_1),
    \end{align*}
    that is, we get the opposite to \eqref{eq:2p1} inequality.
    \end{itemize}
    Note also that since $s_2'(s_1) \in (s_1^{(q-1)/(p-1)},1)$ we have that
    \[
    \lim \limits_{s_1 \to 1^-} s_2'(s_1) = 1
    \]
    Also, since $h_{s_1}(s_2') = t_1^{(1)}(s_1) = \phi(\omega_p(s_1)),$ we have that
    \begin{align*}
        &s_1 - \frac{p}{p-q} \frac{s_1}{s_2'} = \phi(\omega_p(s_1)) < 0 \Rightarrow \\
        &\frac{p}{p-q} \frac{s_1}{s_2'} = s_1 - \phi(\omega_p(s_1)) > 0 \Rightarrow \\
        &s_2' = s_2'(s_1) = \frac{p}{p-q} \frac{s_1}{ s_1 - \phi(\omega_p(s_1))},
    \end{align*}
    that is, we provided the dependence of $s_2'$ by $s_1.$

\pagebreak

\noindent Also, $t_{s_1,s_2'}(0) = \omega_p(s_1). $ Indeed:
\[
\phi(t_{s_1,s_2'}(0)) = h(s_1, s_2') = t_1^{(1)}(s_1) = \phi(\omega_p(s_1))
\]
and the equality above is implied. 

\vspace{5mm}

\noindent Moreover $t_{s_1,s_2'}'(0) \leq 0,$ since otherwise we would have $t_{s_1,s_2'}'(0) > 0,$ and since for $\beta_0 = \omega_p(s_1)-1:$
\[
t_{s_1,s_2'}(\beta_0) \leq t_1(\beta_0) = \omega_p(s_1) = t_{s_1,s_2'}(0),
\]
we get by the Remarks in section 4 of \cite{Nikol} a contradiction.
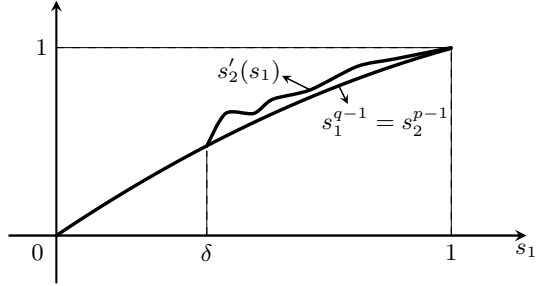
\begin{center}
\begin{tikzpicture}[line cap=round,line join=round,>=stealth,x=2.0cm,y=1.0cm,scale=1.25]
	\draw[->,color=black,line width=0.9pt] (-0.25,0.) -- (2.5,0.);
	\foreach \x in {}
	\draw[shift={(\x,0)},color=black] (0pt,-2.5pt) -- (0pt,2.5pt)
	node[below] at (0,-1pt) {};	
	\draw[->,color=black,line width=0.8pt] (0,-0.5) -- (0.,2.5);
	\foreach \y in {}
	\draw[shift={(0,\y)},color=black] (-2pt,0pt) -- (2pt,0pt) node[left] at (-1pt,0pt) {};
\draw[shift={(0.,0.)},line width=1.3pt,color=black] plot[smooth,samples=200,domain=0.0:2.1,variable=\t] ({\t},{-0.15*(1.12*\t-4)^2+2.4});
	\draw [fill=black,dashed] (0,2) -- (2.1,2);
	\draw [fill=black,dashed] (2.1,0) -- (2.1,2);
	\draw [fill=black,dashed] (0.8,0) -- (0.8,0.9);
	\draw [->,fill=black,line width=0.5pt] (1.5,1.6) -- (1.55,1.4);
	\draw [->,fill=black,line width=0.5pt] (1.35,1.55) -- (1.2,1.7);
	\begin{small}
		\draw[color=black] (0.8,0) node[below] {$\delta$};
		\draw[color=black] (1.24,1.75) node[left] {$s'_2(s_1)$};
		\draw[color=black] (1.75,1.45) node[below] {$s_1^{q-1}=s_2^{p-1}$};
		\draw[color=black] (-0.1,0) node[below] {$0$};
		\draw[color=black] (2.5,0) node[below]  {$s_1$};
		\draw[color=black] (2.1,0) node[below]  {$1$};
		\draw[color=black] (0,2) node[left]  {$1$};
		\draw [line width=1.3pt ] plot [smooth,samples=200] coordinates {(0.8,0.95) (0.9,1.3)(1.05,1.3) (1.15,1.45) (1.35,1.55)  (1.6,1.8) (1.8,1.88) (2.1,2)};
	\end{small}
\end{tikzpicture}
\end{center}
Finally,
\[
s_2'(\delta) = \frac{p}{p-q} \frac{\delta}{\delta - \phi(\omega_p(\delta))}.
\]
But $\delta$ is defined in a way that
\begin{align*}
\vartheta(\delta)&=0 \Leftrightarrow \nonumber \\
\phi(\omega_p(\delta)) &= \delta - \frac{p}{p-q} \delta^{(p-q)/(p-1)} , 
\end{align*}
because of \eqref{eq:2p6}. Thus
\begin{align*}
    &s_2'(\delta) = \frac{p}{p-q} \frac{\delta}{\frac{p}{p-q} \delta^{(p-q)/(p-1)}} = \delta^{(q-1)/(p-1)} \Rightarrow \\
    &(\delta, s_2'(\delta)) \in C_1,
\end{align*}
where $C_1$ is the curve $s_1^{q-1} = s_2^{p-1}$. We conclude the proof of the following: 
\begin{theorem} \label{thm:2p1}
    There exists $0 < \delta = \delta(p,q) < 1$ and a function
    $s_2':[\delta,1) \to [\delta^{(q-1)/(p-1)},1)$ for which $s_2'(s_1)>s_1^{(q-1)/(p-1)}, \forall s_1 \in (\delta,1)$ such that the following hold:
    \begin{enumerate}
        \item $\forall s_1 \in (0,\delta), \forall s_2 \in [s_1^{(q-1)/(p-1)},1),$ we have that $\omega_p(s_1) < t_{s_1,s_2}(0)$

        \item $\forall s_1 \in (\delta ,1), \forall s_2 \in (s_2'(s_1),1):$ $\omega_p(s_1) < t_{s_1,s_2}(0)$

        \item $\forall s_1 \in (\delta ,1), \forall s_2 \in [s_1^{(q-1)/(p-1)},s_2'(s_1)):$ $\omega_p(s_1) > t_{s_1,s_2}(0)$

        \item $t_{s_1,s_2'(s_1)}(0) = \omega_p(s_1), \forall s_1 \in [\delta,1).$
    \end{enumerate}
\end{theorem}

\pagebreak

\noindent Now we search for explicit characteristic conditions on $(s_1,s_2) \in D$ under which we have $t'_{s_1,s_2}(0) \leq 0$ or $t'_{s_1,s_2}(0) > 0.$ \\

\noindent As we shall see, there are always elements $(s_1,s_2) \in D$ for which $t_{s_1,s_2}'(0)>0.$ Recall that:
\[
D=\left \{(s_1,s_2)\in\mb{R}^2\;:\; 0<s_1^{q-1}\leq s_2^{p-1}<1 \right \}
\]
By relation (4.5) of \cite{Nikol} we have that
\begin{align*} \label{eq:2p9}
    t_{s_1,s_2}'(0) &\leq 0 \Leftrightarrow \\
    pt^{p-q}(0) q (q-1) \frac{(-1)}{G^2(0)} &- p(q-1)(p-q) \frac{s_1}{G^2(0)} \\
    +\frac{s_1}{s_2} \bigg [ p(q-1) \frac{p-q}{G^2(0)} \omega_q(s_2)^q &- \frac{pq(q-1)(-1)}{G^2(0)} \bigg ] \leq 0 \Leftrightarrow \\
    -qt^{p-q}(0) - (p-q)s_1 + \frac{s_1}{s_2} &\left [ (p-q) \omega_q(s_2)^q + q \right ] \leq 0 \Leftrightarrow \\
    t^{p-q}(0) \geq \frac{p-q}{q} \frac{s_1}{s_2} \omega_q(s_2)^q &- \frac{p-q}{q} s_1 + \frac{s_1}{s_2} \Leftrightarrow \\
    t^{p-q}(0) \geq \frac{p-q}{q} s_1 \biggl ( \frac{\omega_q(s_2)^q}{s_2} &- 1 \biggr ) + \frac{s_1}{s_2} \Leftrightarrow \\
    t^{p-q}(0) \geq \frac{p-q}{q} s_1 a(s_2) &+ \frac{s_1}{s_2}, \numberthis
\end{align*}
where
\[
a(s_2) = \frac{\omega_q(s_2)^q}{s_2} - 1 \quad \text{and} \quad t(0) = t_{s_1,s_2}(0).
\]
Analogously, $$t_{s_1,s_2}'(0) > 0 \Leftrightarrow t^{p-q}(0) < \frac{p-q}{q} s_1 a(s_2) + \frac{s_1}{s_2}.$$
Remember that $t(0)$ satisfies:
\[
\phi(t(0)) = s_1 - \frac{p}{p-q}\frac{s_1}{s_2} := h(s_1,s_2), \quad t(0)>1,
\]
where $\phi:[1,\infty) \to \mathbb{R}$ is the strictly increasing function:
\[
\phi(y) = y^p - \frac{p}{p-q} y^{p-q} .
\]
Suppose now that there exists $(s_1,s_2) \in D$ for which $t_{s_1,s_2}'(0) > 0.$ Then by the comments above:
\begin{align*}
    &1 < t^{p-q}(0) < \frac{p-q}{q} s_1 a(s_2) + \frac{s_1}{s_2} =: E(s_1,s_2) \Rightarrow \\
    &1< t(0) < E(s_1,s_2)^{1/(p-q)} \Rightarrow \phi(t(0)) < \phi(E(s_1,s_2)^{1/(p-q)}) \Rightarrow
\end{align*}
\begin{align*} \label{eq:2p10}
    &s_1 - \frac{p}{p-q} \frac{s_1}{s_2} < E(s_1,s_2)^{p/(p-q)} - \frac{p}{p-q} E(s_1,s_2) \Leftrightarrow \\
    &s_1 - \frac{p}{p-q} \frac{s_1}{s_2} < E(s_1,s_2)^{p/(p-q)} - \frac{p}{p-q} \left ( \frac{p-q}{q} s_1 a(s_2) + \frac{s_1}{s_2} \right ) \Rightarrow \\
    &s_1 + \frac{p}{q}s_1 a(s_2) < E(s_1,s_2)^{p/(p-q)} \Rightarrow \\
    &E(s_1,s_2)^p > \left ( \frac{p}{q} s_1 a(s_2) + s_1 \right )^{p-q} \Rightarrow \\
    &s_1^p \left ( \frac{p-q}{q} a(s_2) + \frac{1}{s_2} \right )^p > s_1^{p-q} \left ( \frac{p}{q} a(s_2) + 1 \right )^{p-q} \Rightarrow \\
    &h(s_2) := \frac{\left ( \frac{p-q}{q} a(s_2) + \frac{1}{s_2} \right )^p}{\left ( \frac{p}{q} a(s_2) + 1 \right )^{p-q}} > \frac{1}{s_1^q} . \numberthis
\end{align*}
Assume now that $(s_1,s_2) \in D$ satisfies $t_{s_1,s_2}'(0) = 0,$ then according to the previous reasoning we should have $h(s_2) = \frac{1}{s_1^q}$ and if $t_{s_1,s_2}'(0)<0$ then $h(s_2) < \frac{1}{s_1^q}.$ We calculate now the derivative of $a(s_2)$ with respect to $s_2$. We state it as a separate lemma
\begin{lemma} \label{lem:2p2}
    If $a(s_2) = \frac{\omega_q(s_2)^q}{s_2} - 1,$ \ $s_2 \in (0,1)$ then
    \[
    a'(s_2) = \frac{1}{s_2} \left [ - \frac{\omega_q(s_2)}{(q-1)(\omega_q(s_2)-1)} - \frac{\omega_q(s_2)^q}{s_2} \right ] .
    \]
\end{lemma}
\begin{proof}
    \begin{align*}
        a'(s_2) &= \frac{q\omega_q(s_2)^{q-1} \omega_q'(s_2) s_2 - \omega_q(s_2)^q}{s_2^2} \\
        &= \frac{q}{s_2} \omega_q(s_2)^{q-1} \frac{1}{q(q-1)\omega_q(s_2)^{q-2}(1-\omega_q(s_2))} - \frac{\omega_q(s_2)^q}{s_2^2} \\
        &= - \frac{1}{s_2} \frac{1}{q-1} \frac{\omega_q(s_2)}{\omega_q(s_2) - 1} - \frac{\omega_q(s_2)^q}{s_2} \\
        &= \frac{1}{s_2} \left [ - \frac{\omega_q(s_2)}{(q-1)(\omega_q(s_2)-1)} - \frac{\omega_q(s_2)^q}{s_2} \right ] . \qedhere
    \end{align*}
\end{proof}
\vspace{5mm}
\noindent We study now the monotonicity properties of $h(s_2),$ which is defined in \eqref{eq:2p10}. We set $s=s_2$ and evaluate the sign of $h'(s)$ as follows:
\begin{align*}
    &h'(s) \cong p \left ( \frac{p-q}{q} a(s) + \frac{1}{s} \right )^{p-1} \left ( \frac{p-q}{q} a'(s_2) - \frac{1}{s^2} \right ) \left ( \frac{p}{q} a(s) + 1 \right )^{p-q} \\
    &- \left ( \frac{p-q}{q} a(s) + \frac{1}{s} \right )^p (p-q) \left ( \frac{p}{q} a(s) + 1 \right )^{p-q-1} \left ( \frac{p}{q} a'(s) \right )
\end{align*}

\pagebreak

\begin{align*}
    &\cong \left ( \frac{p-q}{q} a'(s) - \frac{1}{s^2} \right ) \left ( \frac{p}{q} a(s) + 1 \right ) - \frac{p-q}{q} a'(s) \left ( \frac{p-q}{q} a(s) + \frac{1}{s} \right ) \\
    &= \frac{p(p-q)}{q^2} a'(s) a(s) + \frac{p-q}{q} a'(s) - \frac{1}{s^2} \frac{p}{q} a(s) - \frac{1}{s^2} \\
    &- \frac{(p-q)^2}{q^2} a'(s) a(s) - \frac{p-q}{q} a'(s) \frac{1}{s} \\
    &= \frac{p-q}{q} a'(s) \left [ \frac{\omega_q(s)^q}{s} - \frac{1}{s} \right ] - \frac{p}{q} \frac{a(s)}{s^2} - \frac{1}{s^2} < 0 , \forall s \in (0,1) ,
\end{align*}

\vspace{5mm}

\noindent Thus $h(s)$ is strictly decreasing for $s \in (0,1)$, while $h(1)=1$ and $\lim \limits_{s \to 0^+}h(s) = +\infty$. This last fact can be proved easily by making the change of variables in $h(s)$, $t = \omega_q(s) \Leftrightarrow s = H_q(t)$ and then observing that $s \to 0^+ \Leftrightarrow t \to {\frac{q}{q-1}}^-$. We are omitting the simple details of this claim.

\vspace{5mm}

\noindent By the above mentioned comments, it is implied that for every $s_1 \in (0,1)$ there exists unique $s_2 \in (0,1)$ for which $h(s_2) = \frac{1}{s_1^q}$ . We denote this real number by
\begin{equation} \label{eq:2p11}
s_2''(s_1) = s_2, \ \text{that is} \ h(s_2''(s_1)) = \frac{1}{s_1^q}, \forall s_1 \in (0,1) .
\end{equation}
We prove now that there do exist $(s_1,s_2) \in D^o$ for which
\[
h(s_2) = \frac{1}{s_1^q} \Leftrightarrow s_2 = h^{-1} \left ( \frac{1}{s_1^q} \right ) =: h_0(s_1)
\]
Note that $h_0(s_1)$ is an increasing function of $s_1 \in (0,1)$ with $h_0(1^-) = 1.$ Assume on the contrary that the claim stated above is not satisfied. Then for every $(s_1,s_2)$ for which $h(s_2) = \frac{1}{s_1^q}$ we should have:
\begin{equation} \label{eq:2p12}
s_2 \leq s_1^{(q-1)/(p-1)} \Leftrightarrow h(s_2) \geq h(s_1^{(q-1)/(p-1)}) \Leftrightarrow h(s_1^{(q-1)/(p-1)}) \leq \frac{1}{s_1^q}
\end{equation}
But as we shall now see \eqref{eq:2p12} fails for $s_1 \to 1^-$. We make the change of variable $s_1^{(q-1)/(p-1)} = s$ in \eqref{eq:2p12}. Then if \eqref{eq:2p12} was true for $s_1 \to 1^-$, we should have that
\begin{equation} \label{eq:2p13}
    h(s) s^{q \frac{p-1}{q-1}} \leq 1 , \ \text{as} \ s \to 1^-.
\end{equation}
Define the left hand side of \eqref{eq:2p13} by
\[
\Psi(s) := h(s) s^{q \frac{p-1}{q-1}} = s^{q \frac{p-1}{q-1}} \left [ \frac{p-q}{q} a(s) + \frac{1}{s} \right ]^p \Bigg / \left ( \frac{p}{q} a(s) + 1 \right )^{p-q} ,
\]

\pagebreak

\noindent and we have $\Psi(s) \leq 1 $ if and only if
\begin{equation} \label{eq:2p14}
    s^{q \frac{p-1}{q-1}} \left [ \frac{p-q}{q} a(s) + \frac{1}{s} \right ]^p - \left [ \frac{p}{q} a(s) + 1 \right ]^{p-q} =: \Psi_1(s) \leq 0
\end{equation}
But after some tedious calculations, we notice that the derivative of $\Psi_1$ at the point 1 is negative, while $\Psi_1(1) = 0$. Thus \eqref{eq:2p14} cannot be true for $s$ close to $1^-.$ 

\vspace{5mm}

\noindent By what we have mentioned before we have now that if $(s_1,s_2) \in D$ such that $t_{s_1,s_2}'(0)<0$ then $h(s_2) < \frac{1}{s_1^q} = h(s_2''),$ thus $s_2 > s_2'' = s''(s_1).$

\vspace{5mm}

\noindent We prove now the reverse direction of this implication. Assume that $1>s_2 > s_2'',$ then $h(s_2) < h(s_2'') = \frac{1}{s_1^q} \Rightarrow h(s_2) s_1^q < 1$ and then by using the equivalent inequalities of \eqref{eq:2p10}, we obtain:
\begin{equation} \label{eq:2p15}
    \phi(t(0)) > \phi(E(s_1,s_2)^{1/(p-q)}).
\end{equation}
We consider now two cases:
\begin{enumerate}[i)]
    \item If $(s_1,s_2)$ satisfies $E(s_1,s_2) > 1$ then by \eqref{eq:2p15} we immediately get $t(0) > E(s_1,s_2)^{1/(p-q)} \Leftrightarrow t(0)^{p-q} > E(s_1,s_2)$ which as we have seen before is equivalent to
    $t_{s_1,s_2}'(0) < 0$.

    \item If $(s_1,s_2)$ satisfies $E(s_1,s_2) \leq 1$ then we obviously have $E(s_1,s_2) \leq 1 < t(0)^{p-q},$ which in turn implies $t_{s_1,s_2}'(0)<0$.
\end{enumerate}
By the above remarks we conclude that for each $(s_1,s_2) \in D$, the inequality $t_{s_1,s_2}'(0)>0$ is equivalent to $s_1^{(q-1)/(p-1)} \leq s_2 < s_2''(s_1) = s_2'',$ where $s_2''$ satisfies: $h(s_2'') = \frac{1}{s_1^q}$.

\vspace{5mm}

\noindent Denote by
\[
X = \left \{ (s_1,s_2) \in D: \ s_1^{(q-1)/(p-1)} \leq s_2 < s_2'' = h^{-1} \left ( \frac{1}{s_1^q} \right ) \right \}
\]
Then $X$ is a non-empty subset of $D$, for which it is satisfied
$$t(0) = t_{s_1,s_2}(0) < \omega_p(s_1), \forall(s_1,s_2) \in X$$
Indeed, if for some $(s_1,s_2) \in X$ we had $t(0) \geq \omega_p(s_1),$ then as remarked previously, we would have the existence of a $\beta \in (0, 1/(p-1))$ such that
$t_{s_1,s_2}(\beta) \leq t(0)$ for which also $t_{s_1,s_2}'(\beta)=0$. Then, this $t=t_{s_1,s_2}(\beta)$ would satisfy
\[
t = \min \left \{ t(\gamma) : \gamma \in \left [0 , \frac{1}{p-1} \right ] \right \} ,
\]

\pagebreak

\noindent which is a contradiction since $(s_1,s_2) \in X$ implies that: $t_{s_1,s_2}'(0)>0$ and as a consequence
\[
t_{s1,s_2}'(\beta') > 0 , \forall \beta' \in (0, 1/(p-1)).
\]
Thus by Theorem \ref{thm:2p1} the set $X$ is contained in the set
\[
Y = \left \{ (s_1,s_2) \in D : s_1^{(q-1)/(p-1)} \leq s_2 < s_2'(s_1) = s_2' \right \}
\]
We write all the above in the following statement.
\begin{theorem}
    There exists a function $h(s_2)$ defined for $s_2 \in (0,1)$ which is strictly decreasing and for which it is satisfied
    \[
    t_{s_1,s_2}'(0) < 0 \Leftrightarrow s_2 > h^{-1} \left ( \frac{1}{s_1^q} \right ) , \forall (s_1,s_2) \in D .
    \]
    Moreover, $h(s_2)$ is given in \eqref{eq:2p10}.
\end{theorem}
\noindent Thus we also have, according to the results of \cite{Nikol}, that if for every $(s_1,s_2) \in D$ we denote by
\begin{equation} \label{eq:2p16}
t(s_1,s_2) = \min \left \{ t_{s_1,s_2}(\beta) = t(s_1,s_2,\beta) : \beta \in \left [0, \frac{1}{p-1} \right ] \right \}
\end{equation}
then, if $(s_1,s_2) \in X \subseteq D$ we have that $t(s_1,s_2) = t_{s_1,s_2}(0)$ which is defined by the equation \[
t_{s_1,s_2}^p(0) - \frac{p}{p-q} t_{s_1,s_2}^{p-q}(0) = h(s_1,s_2) := s_1 - \frac{p}{p-q} \frac{s_1}{s_2} ,
\]
while if $(s_1,s_2) \in D \setminus X$ we have that $t_{s_1,s_2}'(0) \leq 0$ so that $t=t(s_1,s_2)$ which is defined by \eqref{eq:2p16} above satisfies
\[
F_{s_1,s_2}(t) = 0,
\]
where also $t<t(0) = t_{s_1,s_2}(0)$ if $t_{s_1,s_2}'(0)<0$ and $t=t(0)$ if $t_{s_1,s_2}'(0)=0$ and where the function $F_{s_1,s_2}$ is explicitly given in \cite{Nikol}.

\vspace{5mm}

\noindent Note also that in the case where $t_{s_1,s_2}'(0) > 0,$ that is when $(s_1,s_2) \in X,$ the following hold:

\vspace{5mm}

\noindent Obviously by what is mentioned right above we have that $t(0) = t_{s_1,s_2}(0) = t(s_1,s_2)$ and $t(0)$ satisfies
\[
\phi(t(0)) = h(s_1,s_2) := s_1 - \frac{p}{p-q} \frac{s_1}{s_2},
\]
where $\phi(y) = y^p - \frac{p}{p-q} y^{p-q}, \ y\geq 1.$

\pagebreak

\noindent Then $t(0) = \phi^{-1} \left ( h(s_1,s_2) \right )$ and thus:
\begin{equation} \label{eq:2p17}
    \begin{cases}
        (i) \ \frac{\partial t(0)}{\partial s_1} = \left ( \phi^{-1} \right )' \left ( h(s_1,s_2) \right ) \cdot \frac{\partial h}{ \partial s_1} , \\ \\
        (ii) \ \frac{\partial t(0)}{\partial s_2} = \left ( \phi^{-1} \right )' \left ( h(s_1,s_2) \right ) \cdot \frac{\partial h}{ \partial s_2}
    \end{cases}
\end{equation}
Since \[
\left ( \phi^{-1} \right )' \left ( h(s_1,s_2) \right ) = \frac{1}{\phi' \left ( \phi^{-1} \left ( h(s_1,s_2) \right ) \right ) } = \frac{1}{\phi'(t(0))} 
\]
and $t(0) > 1$ (since $s_2 < 1$) we should have by \eqref{eq:2p17} that:
\begin{equation} \label{eq:2p18}
    \begin{cases}
        \frac{\partial t(0) }{\partial s_1} \cong \frac{\partial h}{\partial s_1} = 1 - \frac{p}{p-q} \frac{1}{s_2} < 0 \\ \\
        \frac{\partial t(0)}{\partial s_2} \cong \frac{\partial h}{\partial s_2} = \frac{p}{p-q} \frac{s_1}{s_2^2} > 0
    \end{cases}
    \forall (s_1,s_2) \in X.
\end{equation}

\vspace{50pt}
\noindent Nikolidakis Eleftherios\\
Assistant Professor\\
Department of Mathematics \\
Panepistimioupolis, University of Ioannina, 45110\\
Greece\\
E-mail address: enikolid@uoi.gr

\end{document}